\documentclass[reqno,UKenglish, 12pt]{amsart}
\usepackage[margin=3cm]{geometry}
\usepackage{amsthm} 
\usepackage{amssymb}
\usepackage{amsopn}
\usepackage{color}
\usepackage{babel} 
\usepackage{hyperref}
\usepackage{amsmath, amsfonts, amssymb} 
\usepackage{graphicx}
\allowdisplaybreaks 
\usepackage[utf8]{inputenc}
\title{Hypergraph Laplace Operators for Chemical Reaction Networks}
\author{Jürgen Jost}
\address{Max Planck Institute for Mathematics in the Sciences, D-04103 Leipzig, Germany}
\email{jjost@mis.mpg.de}
\author{Raffaella Mulas}
\address{Max Planck Institute for Mathematics in the Sciences, D-04103 Leipzig, Germany}
\email{raffaella.mulas@mis.mpg.de}

\usepackage{tikz}
\usetikzlibrary{topaths,calc}

\usepackage{amssymb, amsmath}
\usepackage{verbatim}
\usepackage{hyperref}

\theoremstyle{plain}
\newtheorem{theorem}{Theorem}
\newtheorem{lem}[theorem]{Lemma}
\newtheorem{cor}[theorem]{Corollary}

\newtheorem{prop}[theorem]{Proposition}

\theoremstyle{definition}

\newtheorem*{definition}{Definition}

\newtheorem{ex}{Example}

\theoremstyle{remark}

\newtheorem{rmk}[theorem]{Remark}
\newtheorem*{notation}{Notation}
\newtheorem{recall}[theorem]{Recall}

\begin{document}
	\maketitle
	
	\begin{abstract}We generalize the \emph{normalized combinatorial Laplace operator} for graphs by defining two Laplace operators for hypergraphs that can be useful in the study of chemical reaction networks. We also investigate some properties of their spectra.	\end{abstract}
	\section{Introduction}
 At an abstract level, chemical reaction networks can be modelled as directed hypergraphs in which each vertex represents a chemical element and each hyperedge represents a chemical reaction involving the elements that it contains as vertices. In this paper, we therefore define and study a normalized combinatorial Laplace operator for hypergraphs, with the aim of investigating reaction networks through the spectrum of that operator, that is, its collection of eigenvalues.  We already know that the spectrum of the \emph{normalized combinatorial Laplace operator} (that from now on we will just call \emph{Laplace operator}) for graphs encodes important information about the graphs. For example, we know that the multiplicity of the eigenvalue $0$ for the Laplacian on vertices $L^0$ is equal to the number of connected components of the graph; we know that the multiplicity of the eigenvalue $0$ for the Laplacian on edges $L^1$ is equal to the number of cycles; the largest eigenvalue reaches its maximum value exactly for bipartite graphs and its minimum value exactly for complete graphs. While a graph is not completely determined by its spectrum -- there exist \emph{isospectral graphs}, that is, different graphs with the same spectrum --, the spectrum does capture the important qualitative properties of a graph. That is, classifying graphs by their spectrum may ignore some little details, but seems to be quite useful in the presence of big data, in particular since eigenvalue computations can be performed with tools from linear algebra.\newline

	In Section \ref{Section Basic definitions and assumptions} we define the basic definitions regarding the hypergraphs that represent chemical reaction networks and we make some important assumptions motivated by the chemical interpretation. In Section \ref{Section Generalized Laplace Operators} we construct our Laplace operators for hypergraphs by generalizing, in the most natural way, the construction of the graph Laplace operators. We also prove that their restriction to graphs coincides with the well-known graph Laplace operators. In Section \ref{Section First properties} we prove the first basic properties of our Laplace operators; in Section \ref{Section The eigenvalue 0} we talk about the multiplicity of the eigenvalue $0$ for our two Laplacians. In Section \ref{Section Applications of the Min-max Principle} we recall and apply the \emph{Courant-Fischer-Weyl min-max principle} in order to get more insight about the spectra of our Laplacians and, in particular, in Section \ref{Largest eigenvalue} we study the largest eigenvalue: we prove that it reaches its maximum value exactly for \emph{bipartite hypergraphs} and we see when exactly it reaches its minimum value, which is in this case $0$. Finally, in Section \ref{Isospectral hypergraphs}, we talk about \emph{isospectral hypergraphs}.

		\section{Basic definitions and assumptions}\label{Section Basic definitions and assumptions}
As already mentioned in the introduction, chemical reaction networks can be modelled by \emph{directed} hypergraphs. Each reaction is a directed hyperedge, mapping a collection of vertices, its educts or inputs, to another collection, its products or outputs. We could therefore define a suitable Laplace type operator for a directed hypergraph and study its spectrum, as pioneered by F.Bauer \cite{Bauer} for directed graphs. Since such an operator is not self-adjoint w.r.t. some scalar product, however, in general its eigenvalues will not be real, but have nonzero imaginary parts. Here, however, we prefer to work with symmetric operators and real eigenvalues. That would suggest to work with undirected hypergraphs. Nevertheless, we preserve an important bit of additional structure from the chemical reaction networks, the fact that the vertex set of a hyperedge is partitioned into two classes. In the directed case, they correspond to inputs and outputs, but in the setting that we wish to adopt, we do not distinguish these two roles and simply keep the partitioning of the vertices of a hyperedge into two classes. Thus, we are working with hypergraphs with an additional piece of structure, the partitioning of the vertex sets of each hyperedge into two classes. We shall call these \emph{chemical hypergraphs}. 

	\begin{definition}
				A \emph{chemical hypergraph} is a pair $\Gamma=(V,H)$ such that $V=\{v_1,\ldots,v_N\}$ is a finite set of vertices and $H$ is a set such that every element $h$ in $H$ is a pair of elements $(V_h,W_h)$ (input and output, not necessarily disjoint) in $\mathcal{P}(V)\setminus\{\emptyset\}$. The elements of $H$ are called the \emph{oriented hyperedges}. Changing the orientation of a hyperedge $h$ means exchanging its input and output, leading to the pair $(W_h,V_h)$. 
			\end{definition}

					Since every chemical reaction has both educts and products, we  consider only hyperedges that have at least one input and at least one output. 
	\begin{definition}
					A \emph{catalyst} in a hyperedge $h$ is a vertex that is both an input and an output for $h$.
				\end{definition}
	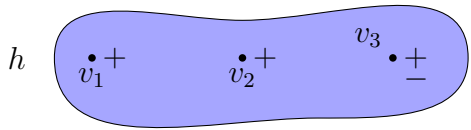
\begin{figure}[h]
		\begin{center}
			\begin{tikzpicture}
			\node (v1) at (1,0) {};
			\node (v2) at (3,0) {};
			\node (v3) at (5,0) {};
			
			\begin{scope}[fill opacity=0.5]
			\filldraw[fill=blue!70] ($(v1)+(-0.5,0)$) 
			to[out=90,in=180] ($(v2) + (0,0.5)$) 
			to[out=0,in=90] ($(v3) + (1,0)$)
			to[out=270,in=0] ($(v2) + (1,-0.8)$)
			to[out=180,in=270] ($(v1)+(-0.5,0)$);
			\end{scope}

			\fill (v1) circle (0.05) node [right] {$+$} node [below] {$v_1$};
			\fill (v2) circle (0.05) node [right] {$+$} node [below] {$v_2$};
			\fill (v3) circle (0.05) node [right] {$+$} node [below right] {$-$} node [above left] {$v_3$};
			
			\node at (0,0) {$h$};
			\end{tikzpicture}
		\end{center}
		\caption{An hyperedge $h$ that has two inputs and one catalyst.}\label{fighyperedgewcat}
	\end{figure}	
	\begin{rmk}
		The above definition comes from the fact that, in chemistry, a \emph{catalyst} is an element that participates in a reaction but is not changed by that reaction.
	\end{rmk}
	
	Our theory thus includes also \emph{oriented graphs with self-loops}, i.e. graphs that may have edges whose two endpoints coincide.\\		

While according to our definition, we shall not work with \emph{directed} hyperedges, we shall nevertheless have to work with \emph{oriented} hyperedges. 
Let us arbitrarily call the two orientations of a hyperedge $h$ $+$ and $-$. Analogously to differential forms in Riemannian geometry, see for instance \cite{JGeom}, we shall consider functions $\gamma$ from the set of oriented hyperedges that satisfy
\begin{equation}
  \label{or}
  \gamma (h,-)=-\gamma (h,+),
\end{equation}
that is, changing the orientation of $h$ produces a minus sign. Importantly, neither of the two orientations that a hyperedge carries plays a preferred role. Thus, an \emph{oriented} hyperedge should not be confused with a \emph{directed} hyperedge. 
		
				\begin{definition}
					We say that a hypergraph $\Gamma=(V,H)$ is \emph{connected} if, for every pair of vertices $v,w\in V$, there exists a \emph{path} that connects $v$ and $w$, i.e. there exist $v_1,\dots,v_m\in V$ and $h_1,\dots,h_{m-1}\in H$ such that:
					\begin{itemize}
						\item $v_1=v$;
						\item $v_m=w$;
						\item $\{v_i,v_{i+1}\}\subseteq h_i$ for each $i=1,\dots,m-1$.
					\end{itemize}
				\end{definition}
				\begin{figure}[h]
					\begin{center}
						\begin{tikzpicture}
						\node (v1) at (1,0) {};
						\node (v2) at (3,0) {};
						\node (v3) at (5,0) {};
						\node (v4) at (7,0) {};
						\node (v5) at (9,0) {};
						
						\begin{scope}[fill opacity=0.5]
						\filldraw[fill=red!70] ($(v1)+(-0.5,0)$) 
						to[out=90,in=180] ($(v2) + (0,0.5)$) 
						to[out=0,in=90] ($(v3) + (1,0)$)
						to[out=270,in=0] ($(v2) + (1,-0.8)$)
						to[out=180,in=270] ($(v1)+(-0.5,0)$);
							\filldraw[fill=blue!70] ($(v3)+(-0.5,0)$) 
							to[out=90,in=180] ($(v4) + (0,0.5)$) 
							to[out=0,in=90] ($(v5) + (1,0)$)
							to[out=270,in=0] ($(v4) + (1,-0.8)$)
							to[out=180,in=270] ($(v3)+(-0.5,0)$);
						\end{scope}

						\fill (v1) circle (0.05) node [below] {$v_1$};
						\fill (v2) circle (0.05) node [below] {$v_2$};
						\fill (v3) circle (0.05) node [below] {$v_3$};
						\fill (v4) circle (0.05) node [below] {$v_4$};
						\fill (v5) circle (0.05) node [below] {$v_5$};
						
						\node at (0,0) {\color{red} $h_1$};
						\node at (10.5,0) {\color{blue} $h_2$};
						\end{tikzpicture}
					\end{center}
					\caption{A \emph{connected} hypergraph.}\label{connhype}
				\end{figure}
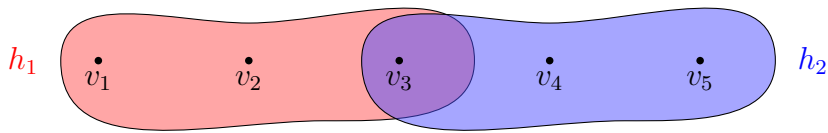
			\begin{definition}
				We say that $\Gamma=(V,H)$ has $k$ \emph{connected components} if there exist $\Gamma_1=(V_1,H_1),\dots,\Gamma_k=(V_k,H_k)$ such that:
				\begin{enumerate}
					\item For every $i\in\{1,\ldots,k\}$, $\Gamma_i$ is a connected hypergraph with $V_i\subseteq V$ and $H_i\subseteq H$;
					\item For every $i,j\in\{1,\ldots,k\}$, $i\neq j$, $V_i\cap V_j=\emptyset$ and therefore also $H_i\cap H_j=\emptyset$.
				\end{enumerate}
			\end{definition}

			\begin{definition} Let $\Gamma=(V,H)$ be a hypergraph. We say that $\mathcal{S}=(V',H')$ is a \emph{closed system of reactions} in $\Gamma$ if:
				\begin{enumerate}
					\item $\emptyset\neq H'\subseteq H$;
					\item $V'=\{v\in h:h\in H'\}$;
					\item Each $v\in V'$ appears in $\mathcal{S}$ as often as input as as output.
				\end{enumerate}
			\end{definition}
			\begin{rmk}
				Closed systems for hypergraphs generalize the \emph{oriented cycles} that we have for graphs, so they are interesting from the mathematical point of view, and they are also clearly interesting from the chemical point of view.
			\end{rmk}

			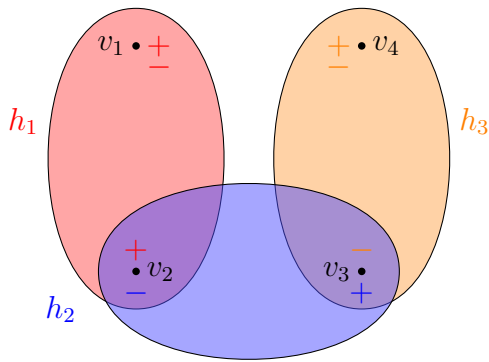
\begin{figure}[h]
				\begin{center}
					\begin{tikzpicture}
					\node (v2) at (1.5,0) {};
					\node (v1) at (1.5,3) {};
					\node (v3) at (4.5,0) {};
					\node (v4) at (4.5,3) {};
					
					\begin{scope}[fill opacity=0.5]
					\filldraw[fill=red!70] ($(v1)+(0,0.5)$) 
					to[out=0,in=0] ($(v2) + (0,-0.5)$)
					to[out=180,in=180] ($(v1)+(0,0.5)$);
					\filldraw[fill=orange!70] ($(v4)+(0,0.5)$) 
					to[out=0,in=0] ($(v3) + (0,-0.5)$)
					to[out=180,in=180] ($(v4)+(0,0.5)$);
					\filldraw[fill=blue!70] ($(v2)+(-0.5,0)$) 
					to[out=-270,in=-270] ($(v3) + (0.5,0)$)
					to[out=270,in=270] ($(v2)+(-0.5,0)$);
					\end{scope}
					
					\fill (v1) circle (0.05) node [left] {$v_1$} node [right] {\color{red}$+$} node [below right] {\color{red}$-$};
					\fill (v2) circle (0.05) node [right] {$v_2$} node [above] {\color{red}$+$} node [below] {\color{blue}$-$};
					\fill (v3) circle (0.05) node [left] {$v_3$} node [below] {\color{blue}$+$} node [above] {\color{orange}$-$};
					\fill (v4) circle (0.05) node [right] {$v_4$} node [left] {\color{orange}$+$} node [below left] {\color{orange}$-$};
					
					\node at (0,2) {\color{red}$h_1$};
					\node at (0.5,-0.5) {\color{blue}$h_2$};
					\node at (6,2) {\color{orange}$h_3$};
					\end{tikzpicture}
				\end{center}
				\caption{A closed system of reactions.}\label{closedsys}
			\end{figure}
				\begin{definition}
					We say that two closed systems $\mathcal{S}=(V',H')$ and $\mathcal{S}=(V',H')$ are \emph{disjoint} if $H\cap H'=\emptyset$.
				\end{definition}
				\begin{rmk}
					Disjoint systems don't have common hyperedges but they may have common vertices.
				\end{rmk}
\begin{definition}
	Let $\Gamma=(V,H)$ be a hypergraph with $M$ hyperedges $h_1,\ldots,h_M$ and $K$ closed systems of reactions $\mathcal{S}_1,\ldots,\mathcal{S}_K$. Let $A=(a_{ij})_{ij}$ be the $K\times M$ matrix such that
	\begin{equation*} 
	a_{ij}:=\begin{cases} 1 & \text{if }h_j\in\mathcal{S}_i\\ 0 & \text{otherwise.} \end{cases} 
	\end{equation*}Therefore each row $A_i$ of $A$ represents a closed system $\mathcal{S}_i$ and each column $A^j$ of $A$ represents a hyperedge $h_j$. Given $I\subseteq\{1,\ldots,K\}$, we say that the closed systems $\{\mathcal{S}_i\}_{i\in I}$ are \emph{linearly independent} if the raws $\{A_i\}_{i\in I}$ of $A$ are linearly independent.
\end{definition}
\begin{rmk}
Pairwise disjoint closed systems are linearly independent.
\end{rmk}

\section{Generalized Laplace Operators}\label{Section Generalized Laplace Operators}
	In order to define the Laplace operator for hypergraphs, we will generalize the construction of the Laplace operator for graphs in the most natural way. In particular, we will:\begin{enumerate}
		\item Give weight one to the hyperedges (as we do for edges in the case of graphs) and therefore give weight $\deg v:=\bigl| \text{hyperedges containing }v\bigr|$ to each vertex $v$;
		\item Define a scalar product for functions defined on hyperedges and a scalar product for functions defined on vertices, based on the weights we gave;
		\item Define the boundary operator for functions defined on the vertex set;
		\item Find the coboundary operator based on the scalar product we defined;
		\item Define the Laplace operators as the two different compositions of the boundary and the coboundary operator.
	\end{enumerate}
\begin{definition}[Scalar product for functions defined on hyperedges]
	Given $\omega,\gamma:H\rightarrow\mathbb{R}$, let
	\begin{equation*}
	(\omega,\gamma)_H:=\sum_{h\in H}\omega(h)\cdot\gamma(h).
	\end{equation*}
\end{definition}
\begin{definition}[Scalar product for functions defined on vertices]
	Given $f,g:V\rightarrow\mathbb{R}$, let
	\begin{equation*}
	(f,g)_V:=\sum_{v\in V}\deg v\cdot f(v)\cdot g(v).
	\end{equation*}
\end{definition}
\begin{definition}[Boundary operator for functions defined on vertices]
	Given $f:V\rightarrow\mathbb{R}$ and $h\in H$, let
	\begin{equation*}
\delta f(h):=\sum_{v_i \text{ input of }h}f(v_i)-\sum_{v^j \text{ output of }h}f(v^j).
	\end{equation*}\end{definition}
\begin{rmk}
	Note that
	\begin{equation*}
	\delta:\{f:V\rightarrow\mathbb{R}\}\longrightarrow\{\gamma:H\rightarrow\mathbb{R}\}
	\end{equation*}
where the $\gamma$ are always supposed to satisfy \eqref{or}. In particular, $\delta f$ also satisfies \eqref{or}. 
\end{rmk}

\begin{definition}[Adjoint operator of the boundary operator]
	Let 	\begin{equation*}
	\delta^*:\{\gamma:H\rightarrow\mathbb{R}\}\longrightarrow\{f:V\rightarrow\mathbb{R}\}
	\end{equation*}be defined as
		\begin{equation*}
		\delta^*(\gamma)(v):=\frac{\sum_{h_{\text{in}}: v\text{ input}}\gamma(h_{\text{in}})-\sum_{h_{\text{out}}: v\text{ output}}\gamma(h_{\text{out}})}{\deg v}.
		\end{equation*}
\end{definition}
\begin{lem}$\delta^*$ is such that $(\delta f,\gamma)_H=(f,\delta^*\gamma)_V$, therefore it is the (unique) adjoint operator of $\delta$.
\end{lem}
\begin{proof}
\begin{align*}
(\delta f,\gamma)_H&=\sum_{h\in H}\gamma(h)\cdot\biggl(\sum_{v_i \text{ input of }h}f(v_i)-\sum_{v^j \text{ output of }h}f(v^j)\biggr)\\
&=\sum_{v\in V}f(v)\cdot\biggl(\sum_{h_{\text{in}}: v\text{ input}}\gamma(h_{\text{in}})-\sum_{h_{\text{out}}: v\text{ output}}\gamma(h_{\text{out}})\biggr)\\
&=\sum_{v\in V}\deg v\cdot f(v)\cdot\frac{\biggl(\sum_{h_{\text{in}}: v\text{ input}}\gamma(h_{\text{in}})-\sum_{h_{\text{out}}: v\text{ output}}\gamma(h_{\text{out}})\biggr)}{\deg v}\\
&=\sum_{v\in V}\deg v\cdot f(v)\cdot\delta^*(\gamma)(v)\\
&=(f,\delta^*\gamma)_V.
\end{align*}
\end{proof}
\begin{definition}[Laplace operators]
	Given $f:V\rightarrow\mathbb{R}$ and given $v\in V$, let
\begin{align*}
L^Vf(v):=&\delta^*(\delta f)(v)\\
=&\frac{\sum_{h_{\text{in}}: v\text{ input}}\delta f(h_{\text{in}})-\sum_{h_{\text{out}}: v\text{ output}}\delta f(h_{\text{out}})}{\deg v}\\
=&\frac{\sum_{h_{\text{in}}: v\text{ input}}\biggl(\sum_{v' \text{ input of }h_{\text{in}}}f(v')-\sum_{w' \text{ output of }h_{\text{in}}}f(w')\biggr)}{\deg v}+\\
&-\frac{\sum_{h_{\text{out}}: v\text{ output}}\biggl(\sum_{\hat{v} \text{ input of }h_{\text{out}}}f(\hat{v})-\sum_{\hat{w} \text{ output of }h_{\text{out}}}f(\hat{w})\biggr)}{\deg v}.
\end{align*}Analogously, given $\gamma:H\rightarrow\mathbb{R}$ and $h\in H$, let
\begin{align*}
L^H\gamma(h):=&\delta(\delta^* \gamma)(h)\\
=&\sum_{v_i \text{ input of }h}\delta^* \gamma(v_i)-\sum_{v^j \text{ output of }h}\delta^* \gamma(v^j)\\
=&\sum_{v_i \text{ input of }h}\frac{\sum_{h_{\text{in}}: v_i\text{ input}}\gamma(h_{\text{in}})-\sum_{h_{\text{out}}: v_i\text{ output}}\gamma(h_{\text{out}})}{\deg v_i}+\\
&-\sum_{v^j \text{ output of }h}\frac{\sum_{h'_{\text{in}}: v^j\text{ input}}\gamma(h'_{\text{in}})-\sum_{h'_{\text{out}}: v^j\text{ output}}\gamma(h'_{\text{out}})}{\deg v^j}.
\end{align*}
\end{definition}
\begin{prop}
Let $\Gamma$ be a graph with vertex set $V$ and edge set $E$, with the  convention for orientations as introduced above for hypergraphs. Then
\begin{equation*}
L^Vf(v)=f(v)-\frac{1}{\deg v}\sum_{v\rightarrow w}f(w),
\end{equation*}which is exactly the Laplace operator of graphs for functions defined on vertices.\newline
 Analogously, if $\gamma:E\rightarrow\mathbb{R}$ is such that $\gamma(-e)=-\gamma(e)$ and $e=[v_0,v_1]$,
 \begin{equation*}
 	L^H\gamma(e)=\frac{1}{\deg v_0}\cdot \sum_{v_o\in f=[v_0,w]}\gamma(f)-\frac{1}{\deg v_1}\cdot \sum_{v_1\in g=[v_1,w]}\gamma(g),
 	\end{equation*}which is equal to $L^1$ for graphs. 
\end{prop}
\begin{proof}
Every oriented edge has exactly one input and exactly one output. Therefore, if $h_{\text{in}}$ is an edge with input $v$, then \begin{equation*}
\{v':v' \text{ input of }h_{\text{in}}\}=\{v\}
\end{equation*}and
\begin{equation*}
\bigl|\{w':w' \text{ output of }h_{\text{in}}\}\bigr|=1.
\end{equation*}
If $h_{\text{out}}$ is an edge with output $v$, then \begin{equation*}
\{\hat{w}:\hat{w} \text{ output of }h_{\text{out}}\}=\{v\}
\end{equation*}and
\begin{equation*}
\big|\{\hat{v}:\hat{v} \text{ input of }h_{\text{out}}\}\big|=1.
\end{equation*}
Therefore,
\begin{align*}
L^Vf(v)=&\frac{\sum_{h_{\text{in}}: v\text{ input}}\biggl(\sum_{v' \text{ input of }h_{\text{in}}}f(v')-\sum_{w' \text{ output of }h_{\text{in}}}f(w')\biggr)}{\deg v}+\\
&-\frac{\sum_{h_{\text{out}}: v\text{ output}}\biggl(\sum_{\hat{v} \text{ input of }h_{\text{out}}}f(\hat{v})-\sum_{\hat{w} \text{ output of }h_{\text{out}}}f(\hat{w})\biggr)}{\deg v}\\
=&\frac{\sum_{h_{\text{in}}: v\text{ input}}\biggl(f(v)-\sum_{w' \text{ output of }h_{\text{in}}}f(w')\biggr)}{\deg v}+\\
&-\frac{\sum_{h_{\text{out}}: v\text{ output}}\biggl(\sum_{\hat{v} \text{ input of }h_{\text{out}}}f(\hat{v})-f(v)\biggr)}{\deg v}\\
=&\frac{f(v)}{\deg v}\cdot\biggl(\big|h_{\text{in}}: v\text{ input}\big|+\big|h_{\text{out}}: v\text{ output}\big|\biggr)+\\
&-\frac{1}{\deg v}\cdot\biggl(\sum_{h_{\text{in}: v \text{ input, }w' \text{ output}}}f(w')+\sum_{h_{\text{out}: \hat{v} \text{ input, }v \text{ output}}}f(\hat{v})\biggr)\\
=&f(v)-\frac{1}{\deg v}\cdot\sum_{v\rightarrow w}f(w),
\end{align*}where the last equality is due to the properties of orientation for graphs. Analogously, if $e=[v_0,v_1]$, then
\begin{align*}
L^H\gamma(e)=&\sum_{v_i \text{ input of }e}\frac{\sum_{h_{\text{in}}: v_i\text{ input}}\gamma(h_{\text{in}})-\sum_{h_{\text{out}}: v_i\text{ output}}\gamma(h_{\text{out}})}{\deg v_i}+\\
&-\sum_{v^j \text{ output of }e}\frac{\sum_{h'_{\text{in}}: v^j\text{ input}}\gamma(h'_{\text{in}})-\sum_{h'_{\text{out}}: v^j\text{ output}}\gamma(h'_{\text{out}})}{\deg v^j}\\
=&\frac{\sum_{h_{\text{in}}: v_0\text{ input}}\gamma(h_{\text{in}})-\sum_{h_{\text{out}}: v_0\text{ output}}\gamma(h_{\text{out}})}{\deg v_0}+\\
&-\frac{\sum_{h'_{\text{in}}: v_1\text{ input}}\gamma(h'_{\text{in}})-\sum_{h'_{\text{out}}: v_1\text{ output}}\gamma(h'_{\text{out}})}{\deg v_1}\\
=&\frac{1}{\deg v_0}\cdot \sum_{v_o\in f=[v_0,w]}\gamma(f)-\frac{1}{\deg v_1}\cdot \sum_{v_1\in g=[v_1,w]}\gamma(g),
\end{align*}where the last equality is due to the fact that $-\gamma(h_{\text{out}})=\gamma(-h_{\text{out}})$ and $-\gamma(h'_{\text{out}})=\gamma(-h'_{\text{out}})$.
\end{proof}
\begin{rmk}
$L^H\gamma(h)$ counts what flows out at the inputs - what flows in at the inputs - what flows out at the outputs + what flows in at the outputs.
\end{rmk}
\section{First properties}\label{Section First properties}
\begin{lem}\label{lem1}
$L^V$ and $L^H$ are both self-adjoint.
\end{lem}
\begin{proof}
Use the fact that $L^V$ and $L^H$ are the two compositions of $\delta$ and $\delta^*$, which are adjoint to each other.
\end{proof}
\begin{lem}\label{lem2}
$L^V$ and $L^H$ are non-negative operators.
\end{lem}
\begin{proof}
Let $f:V\rightarrow\mathbb{R}$. Then
\begin{equation}\label{pos1}
(L^Vf,f)_V=(\delta^*\delta f,f)_V=(\delta f,\delta f)_H\geq 0.
\end{equation}Analogously, for $\gamma:H\rightarrow\mathbb{R}$,
\begin{equation}\label{pos2}
(L^H\gamma,\gamma)_H=(\delta\delta^* \gamma,\gamma)_H=(\delta^* \gamma,\delta^*\gamma)_V\geq 0.
\end{equation}
\end{proof}
A direct consequence of Lemmas \ref{lem1} and \ref{lem2} is 
	\begin{cor}\label{cor1}
		The eigenvalues of $L^V$ and $L^H$ are real and non-negative.
	\end{cor}

	\begin{notation}Let $N:=\big|V\big|$ and let $M:=\big|H\big|$. Since the space of real functions on a set with cardinality $k$ is $k$-dimensional, an operator on this space has precisely $k$ eigenvalues, counted with their multiplicities. Therefore $L^V$ has $N$ eigenvalues that we will arrange as
		\begin{equation*}
		\mu_1\geq\ldots\geq\mu_N.
		\end{equation*}Analogously, $L^H$ has $M$ eigenvalues that we will arrange as
			\begin{equation*}
					\mu_1^H\geq\ldots\geq\mu^H_M.
			\end{equation*}
	\end{notation}

	\begin{lem}
		If $A$ and $B$ are linear operators, then the non-zero eigenvalues of $AB$ and $BA$ are the same.
	\end{lem}
	\begin{proof}Let $\mu$ be a non-zero eigenvalue of $AB$ for a non-zero eigenvector $v$. Then
		\begin{equation*}
		\mu Bv=B\mu v=B(ABv)=(BA)Bv.
		\end{equation*}Therefore, $\mu$ is an eigenvalue of $BA$ for the eigenvector $Bv$.
	\end{proof}
	\begin{cor}\label{cor2}
		The non-zero eigenvalues of $L^V$ and $L^H$ are the same. In particular, if $f$ is an eigenfunction of $L^V$ with eigenvalue $\mu \neq 0$, then $\delta f$ is an eigenfunction of $L^H$ with eigenvalue $\mu$; if $\gamma$ is an eigenfunction of $L^H$ with eigenvalue $\mu'\neq 0$, then $\delta^* \gamma$ is an eigenfunction of $L^V$ with eigenvalue $\mu'$.
	\end{cor}
This corollary is quite important because it offers us two alternative ways to control or estimate the nonvanishing eigenvalues, either through $L^V$ or through $L^H$. In particular, we shall see in Section \ref{Section Applications of the Min-max Principle} below that these eigenvalues can therefore be expressed in two different ways by Rayleigh quotients. \\

As another important consequence of Cor. \ref{cor2}, the two operators only differ in the multiplicity of the eigenvalue $0$. 
Let $m_V$ and $m_H$ be the multiplicity of the eigenvalue $0$ of $L^V$ and $L^H$, resp. Then Cor. \ref{cor2} implies
\begin{cor}\label{cor3}
  \begin{equation}
    \label{mult}
    m_V-m_H =|V|-|H|.
  \end{equation}
\end{cor}

\section{The eigenvalue $0$}\label{Section The eigenvalue 0}
In this section, we want to control the multiplicity of the eigenvalue $0$ for our two Laplacians. They are related by Cor. \ref{cor3}. In order to see the principle, let us start with the simple situation where we only have a set $V$ of vertices, but no (hyper)edges connecting them. Then \eqref{mult} tells us that $m_V=|V|$, which of course can be trivially verified. Now let us add edges. As long as these edges do not form cycles, that is, as long as the graph is a forest, i.e., a collection of trees, we have $m_H=0$, and therefore, each new edge reduces the number of components as well as $m_V=|V|-|H|$ by $1$. When, however, a new edge closes a cycle, then $m_H$ increases by $1$, and consequently, $m_V$ is left unchanged. A special case of this is when we add a loop to a vertex. A loop induces a new eigenvalue $0$ of $L^H$ and thus lets $m_V$ unchanged. The general formula says that $m_V-m_H$ equals the number of connected components minus the number of independent cycles, including self-loops. 

Something analogous happens when we more generally add hyperedges. In contrast to the case of graphs, however, by adding hyperedges, we can potentially eliminate all eigenvalues $0$ of $L^V$. For a graph, $L^V$ always has the eigenvalue $0$, as should be clear from the preceding or also follows from Lemma \ref{lemmabalance} below. We shall see examples of hypergraphs where $L^V$ has only positive eigenvalues. But let us first make some obvious observations.
\begin{lem}\label{lemmacatalyst}
  On a hypergraph with a single hyperedge, $L^V$ has $0$ as an eigenvalue. More precisely, $m_V=|V|-1$ if not every vertex is a catalyst and $m_V=|V|$ if every vertex is a catalyst. 
\end{lem}
\begin{proof}
In Example \ref{ex1} we shall see that, on a hypergraph with a single hyperedge, the only eigenvalue of $L^H$ is non-zero if and only if not every vertex is a catalyst. Therefore, by (\ref{mult}), $m_V=|V|-1$ if not every vertex is a catalyst and $m_V=|V|$ otherwise.
\end{proof}

In order to investigate this in more detail, we observe that by \eqref{pos1}, a function $f$ on the vertex set satisfies $L^Vf=0$   if and only  for every $h\in H$,
					\begin{equation}\label{eigenf0}\sum_{v_i \text{ input of }h}f(v_i)=\sum_{v^j \text{ output of }h}f(v^j).\end{equation}
Thus, to create an eigenvalue $0$ of $L^V$, we need a function  $f:V\rightarrow \mathbb{R}$ such that is not identically $0$ and satisfies \eqref{eigenf0}. \\

Similarly, by \eqref{pos2}, in order to get an eigenvalue $0$ of $L^H$, we need $\gamma: H \to \mathbb{R}$ satisfying \eqref{or} and
\begin{equation}\label{mult0LH}
\sum_{h_{\text{in}}: i\text{ input}}\gamma(h_{\text{in}})=\sum_{h_{\text{out}}: i\text{ output}}\gamma(h_{\text{out}})
\end{equation}
for every vertex $i$. \\
And the multiplicity of the eigenvalue $0$ of $L^V$ and $L^H$ then is given by the number of linearly independent, or since we have scalar products, equivalently by the number of orthogonal $f$ and $\gamma$, resp.,  satisfying these equations. Conversely, if there is no such $f$ or $\gamma$, then the corresponding multiplicity is $0$. \\

 For instance, (\ref{eigenf0}) already implies
\begin{lem}\label{lemmaalways}
  If a hypergraph has a vertex $v_0$ that is a catalyst for every hyperedge that it is contained in, then $L^V$ has $0$ as an eigenvalue.
\end{lem}
\begin{proof}
 Let $f(v_0)=1$ and $f(v)=0$ for $v\neq v_0$. This then satisfies \eqref{eigenf0}. 
\end{proof}
	\begin{rmk}
	Any function $f:V\rightarrow\mathbb{R}$ is an eigenfunction for the eigenvalue $0$ in some hypergraph that has vertex set $V$.  Construct a hypergraph $\Gamma$ in which all the vertices $v_1,\ldots,v_k$ of $V$ are always catalysts. Then $f$ satisfies (\ref{eigenf0}) for $\Gamma$.
					\end{rmk}
In fact, we have 
\begin{prop}\label{charlambdaN}If $k$ vertices are always catalysts, then $m_V \ge k$. \\
And $m_V=N$, or equivalently, $\mu_1=0$, that is, $0$ is the only eigenvalue $ \Longleftrightarrow$ all vertices are always catalysts.
				\end{prop}
				\begin{proof}The first part and the implication $\Longleftarrow$ are clear from the proof of Lemma \ref{lemmaalways}. Let's prove $\Longrightarrow$. In particular, let's assume that there exists at least one vertex $\hat{v}\in\hat{h}$ which is not a catalyst for $\hat{h}$ (without loss of generality, we can assume that it is an input). We want to prove that $\mu_1>0$. Let $f:V\rightarrow\mathbb{R}$ such that $f(w)=0$ for all $w\neq \hat{v}$ and such that \begin{equation*}
					f(\hat{v})=\frac{1}{\sqrt{\deg \hat{v}}}.
					\end{equation*}Then $\sum_{v\in V}\deg v\cdot f(v)^2=1$ and
					\begin{align*}
					&\sum_{h\in H}\biggl(\sum_{v_i \text{ input of }h}f(v_i)-\sum_{v^j \text{ output of }h}f(v^j)\biggr)^2\\
					&\geq \biggl(\sum_{v_i \text{ input of }\hat{h}}f(v_i)-\sum_{v^j \text{ output of }\hat{h}}f(v^j)\biggr)^2\\
					&=\biggl(\frac{1}{\sqrt{\deg \hat{v}}}\biggr)^2\\
					&>0.
					\end{align*}Therefore, $\mu_1>0$.
				\end{proof}
				\begin{rmk}\label{lambda2>0}
					The previous proposition implies that, unlike the case of the graphs, the multiplicity of the eigenvalue $0$ for $L^V$ is in general not equal to the number of connected components of the hypergraph (in particular, we don't have that $\mu_{N-1}> 0$ for every connected hypergraph) and, analogously, the multiplicity of the eigenvalue $0$ for $L^H$ does not count, in general, the cycles of the hypergraph.
				\end{rmk}							
				We shall now see some further special cases of hypergraphs with $\mu_N=0$.
				\begin{lem}\label{lemmabalance}
					Let $\Gamma$ satisfy 
					\begin{equation}
					\label{balance}
					\big|\text{inputs of }h\big|=\big|\text{outputs of }h\big| \text{  for each }h\in H.
					\end{equation}
					Then $L^V$ has the eigenvalue $0$.
				\end{lem}
				This holds in particular for graphs, because there, every edge has precisely one input and one output. 
				\begin{proof}
					When \eqref{balance} holds, then any constant function satisfies \eqref{eigenf0}. 
				\end{proof}
				\begin{rmk}	In fact, some such condition is necessary. 
					More precisely, the fact that $\mu_N=0$ for a hypergraph means that we can give a \emph{weight} $f:V\rightarrow\mathbb{R}$ to the vertices such that, in each hyperedge, inputs and outputs have in total the same weight.
									
									\begin{prop}\label{lambda1=0}
										If $\Gamma$ is one of the following hypergraphs, then $\mu_N=0$:
										\begin{enumerate}
											\item $\Gamma$ is given by the union of a hypergraph $\Gamma'$ with $\mu'_N=0$ together with a hyperedge $h$ such that there exists at least one $v\in h\setminus\Gamma'$;
											\item $\Gamma$ is given by the union of a hypergraph $\Gamma'$ with $\mu'_N=0$ together with a hyperedge $h$ that involves only vertices of $\Gamma'$ and has only catalysts.
										\end{enumerate}
									\end{prop}
									\begin{proof}
										\begin{enumerate}
											\item Assume that $\Gamma$ is given by the union of a hypergraph $\Gamma'$ with $\mu'_N=0$ together with a hyperedge $h$ which involves at least one vertex that is not in $\Gamma'$. Since $\mu'_N=0$, there exists a function $f'$ for $\Gamma'$ that satisfies (\ref{eigenf0}). If there is a vertex in $h\setminus\Gamma'$ which is a catalyst, we can apply Lemma \ref{lemmaalways}. If $h$ involves at least one vertex $\hat{v}\notin\Gamma'$ which is not a catalyst, let
											\begin{equation*}
											f(v):=f'(v)
											\end{equation*}for every $v\in\Gamma'$;
											\begin{equation*}
											f(w):=0
											\end{equation*}for every vertex $w\in h\setminus\Gamma'$, $w\neq \hat{v}$;
											\begin{equation*}
											f(\hat{v}):=\sum_{v^j\in\Gamma': v^j \text{ output of }h}f'(v^j)-\sum_{v_i \in\Gamma': v_i \text{ input of }h}f'(v_i)
											\end{equation*}if $\hat{v}$ is an input and not an output;
											\begin{equation*}
											f(\hat{v}):=\sum_{v_i\in\Gamma': v_i \text{ input of }h}f'(v_i)-\sum_{v^j\in\Gamma': v^j \text{ output of }h}f'(v^j)
											\end{equation*}if $\hat{v}$ is an output and not an input.\newline
											Then $f$ satisfies (\ref{eigenf0}).
											\item Assume that $\Gamma$ is given by the union of a hypergraph $\Gamma'$ with $\mu'_N=0$ together with a hyperedge $h$ which involves only vertices of $\Gamma'$ and which has only catalysts. Since $\mu'_N=0$, there exists a function $f'$ for $\Gamma'$ that satisfies (\ref{eigenf0}). Such $f'$ satisfies (\ref{eigenf0}) also for $\Gamma$.
										\end{enumerate}
									\end{proof}
					We shall now see two examples of hypergraphs with $\mu_N>0$, that is, where $L^V$ does not have $0$ as an eigenvalue. 
				\end{rmk}
				\begin{lem}\label{lambda1>0}
					Let $\Gamma$ be the union of a connected graph $\Gamma'$ with a hyperedge $h$ that involves only vertices of $\Gamma'$ and such that $\big|\text{inputs of }h\big|\neq\big|\text{outputs of }h\big|$. Then $\mu_N>0$.
				\end{lem}
				\begin{proof}We know that $f$ satisfies (\ref{eigenf0}) on a connected graph $\Gamma'$ if and only if $f$ is a constant function. But a constant function $f$ can clearly not satisfy (\ref{eigenf0}) for a hyperedge $h$ such that $\big|\text{inputs of }h\big|\neq\big|\text{outputs of }h\big|$. Therefore, $\mu_N$ can not be $0$ in this case.
				\end{proof}
				\begin{lem}
					Let $\Gamma$ be the hypergraph on $N>2$ vertices $v_1,\ldots,v_N$ with $N$ hyperedges $h_1,\ldots,h_N$ such that, for each $i\in\{1,\ldots,N\}$, $h_i$ has:\begin{itemize}
						\item $v_i$ as input, and
						\item every $v_j$ with $j\neq i$ as output.
					\end{itemize} Then $\mu_N>0$.
				\end{lem}
				\begin{proof}Let $f:V\rightarrow\mathbb{R}$ be a function that satisfies (\ref{eigenf0}). Then for every $i,l\in\{1,\ldots,N\}$,\begin{equation*}
					f(v_i)=\sum_{j\neq i}f(v_j)=f(v_l)+\sum_{j\neq i,l}f(v_j)=f(v_i)+2\cdot \sum_{j\neq i,l}f(v_j).
					\end{equation*}Therefore $\sum_{j\neq i,l}f(v_j)=0$ and $f(v_i)=f(v_l)$. Since this is true for every $i,l\in\{1,\ldots,N\}$, $f$ must be the zero function. This implies that $\mu_N>0$.
				\end{proof}

Now let us see how to apply (\ref{mult0LH}). First, when we have a closed system of reactions, we can take a $\gamma$ that has the same nonzero value on all hyperedges involved in that system and vanishes on all other hyperedges. Such a $\gamma$ then satisfies \eqref{mult0LH} because every vertex in such a system appears the same number of times as input as as output for hyperedges belonging to that system. This is formalized in the next Lemma. 
		
				\begin{lem}\label{lemmaclosedsys}
			If $\Gamma$ has a closed system of reactions, then $\mu_M^H=0$.
				\end{lem}
				\begin{proof}
					Let $\mathcal{S}=(V',H')$ be a closed system in $\Gamma$. Let $\gamma: H\rightarrow\mathbb{R}$ be defined as $\gamma(h'):=1$ for all $h'\in H'$ and $\gamma(h):=0$ for all $h\in H\backslash H'$. Then $\gamma$ satisfies (\ref{mult0LH}). Therefore $\mu_M^H=0$.
				\end{proof}
								\begin{rmk}
									The claim of Lemma \ref{lemmaclosedsys} is actually an \emph{if and only if} for both the case of graphs (for which we know that the multiplicity of $0$ for $L^H$ is equal to the number of oriented cycles) and the case of $\Gamma$ containing only  a single hyperedge. In fact, as we shall see  in Example \ref{ex1}, in this case $\mu_M^H = 0$ if and only if all vertices are catalysts, that is, if and only if there is a closed system of reactions in $\Gamma$ (which is $\Gamma$ itself). But  Example \ref{examplethm22} will show that the converse of Lemma \ref{lemmaclosedsys} does not hold. 
								\end{rmk}
In order to prepare that example, we shall first present another example of a closed system of reactions
\begin{ex}\label{examplethm22a}
Consider a hypergraph with three vertices $v_1,v_2,v_3$, with a hyperedge $h_1$ with input  $v_1$ and output  $v_1, v_2$ and another hyperedge $h_2$  with input  $v_2, v_3$ and output $v_3$.  Thus, $v_1$ and $v_3$ are catalysts. In this system, $v_2$ is created in $h_1$ with the help of $v_1$, without using up $v_1$, and it is destroyed in $h_2$ with the help of $v_3$, without creating anything. Each vertex appears once as input and once as output, and thus, this hypergraph represents a closed system of reactions in the sense of the definition. We shall call this a \emph{source-sink system}.
\end{ex}
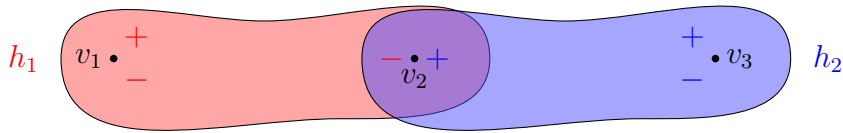
\begin{figure}[h]
	\begin{center}
		\begin{tikzpicture}
		\node (v1) at (1,0) {};
		\node (v2) at (3,0) {};
		\node (v3) at (5,0) {};
		\node (v4) at (7,0) {};
		\node (v5) at (9,0) {};
		
		\begin{scope}[fill opacity=0.5]
		\filldraw[fill=red!70] ($(v1)+(-0.7,0)$) 
		to[out=90,in=180] ($(v2) + (0,0.5)$) 
		to[out=0,in=90] ($(v3) + (1,0)$)
		to[out=270,in=0] ($(v2) + (1,-0.8)$)
		to[out=180,in=270] ($(v1)+(-0.7,0)$);
		\filldraw[fill=blue!70] ($(v3)+(-0.7,0)$) 
		to[out=90,in=180] ($(v4) + (0,0.5)$) 
		to[out=0,in=90] ($(v5) + (1,0)$)
		to[out=270,in=0] ($(v4) + (1,-0.8)$)
		to[out=180,in=270] ($(v3)+(-0.7,0)$);
		\end{scope}

		\fill (v1) circle (0.05) node [left] {$v_1$} node [above right] {\color{red}$+$} node [below right] {\color{red}$-$};
		\fill (v3) circle (0.05) node [below] {$v_2$} node [right] {\color{blue}$+$} node [left] {\color{red}$-$};
		\fill (v5) circle (0.05) node [right] {$v_3$} node [above left] {\color{blue}$+$} node [below left] {\color{blue}$-$};
		
		\node at (-0.2,0) {\color{red} $h_1$};
		\node at (10.5,0) {\color{blue} $h_2$};
		\end{tikzpicture}
	\end{center}
	\caption{The hypergraph in Example \ref{examplethm22a}.}
\end{figure}
We shall now use this principle to create another example that is no longer a closed system of reactions, but makes use of the possibility demonstrated in the previous example to create and destroy products independently. And this will allow us to let the system branch and reunite in between. 
								\begin{ex}\label{examplethm22}
									Let $\Gamma$ be the hypergraph with $4$ vertices $v_1,\ldots,v_4$ and $3$ hyperedges $h_1,h_2,h_3$ such that:
									\begin{enumerate}
										\item $h_1$ has $v_1$ as input and $v_2$ as output;
										\item $h_2$ has $v_1$ as output and $v_3$ as catalyst;
										\item $h_3$ has $v_1$ as input, $v_2$ as input and $v_4$ as catalyst.
									\end{enumerate}
									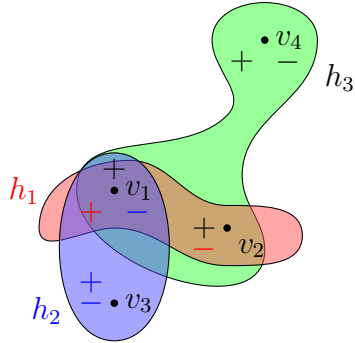
\begin{figure}[h]
										\begin{center}
											\begin{tikzpicture}
											\node (v4) at (4,2.5) {};
											\node (v1) at (2,0.5) {};
											\node (v2) at (3.5,0) {};
											\node (v3) at (2.0,-1) {};
											
											\begin{scope}[fill opacity=0.5]
											\filldraw[fill=green!80] ($(v1)+(-0.5,0)$)
											to[out=90,in=225] ($(v4)+(-0.5,-1)$)
											to[out=45,in=270] ($(v4)+(-0.7,0)$)
											to[out=90,in=180] ($(v4)+(0,0.5)$)
											to[out=0,in=90] ($(v4)+(0.7,0)$)
											to[out=270,in=90] ($(v4)+(-0.3,-1.8)$)
											to[out=270,in=90] ($(v2)+(0.5,-0.3)$)
											to[out=270,in=270] ($(v1)+(-0.5,0)$);
											\filldraw[fill=red!70] ($(v1)+(-1,-0.5)$) 
											to[out=90,in=180] ($(v1) + (0.2,0.4)$) 
											to[out=0,in=180] ($(v2) + (0,0.3)$)
											to[out=0,in=90] ($(v2) + (1,-0.1)$)
											to[out=270,in=0] ($(v2) + (0,-0.5)$)
											to[out=180,in=0] ($(v2) + (-1.5,0)$)
											to[out=180,in=270] ($(v1)+(-1,-0.5)$);
											\filldraw[fill=blue!70] ($(v1)+(0,0.5)$) 
											to[out=0,in=0] ($(v3) + (0,-0.5)$)
											to[out=180,in=180] ($(v1)+(0,0.5)$);
											\end{scope}
											
											\fill (v1) circle (0.05) node [right] {$v_1$} node [below left] {\color{red}$+$} node [below right] {\color{blue}$-$} node [above] {$+$};
											\fill (v2) circle (0.05) node [below right] {$v_2$} node [below left] {\color{red}$-$} node [left] {$+$};
											\fill (v3) circle (0.05) node [right] {$v_3$} node [above left] {\color{blue}$+$} node [left] {\color{blue}$-$};
											\fill (v4) circle (0.05) node [right] {$v_4$} node [below right] {$-$} node [below left] {$+$};
											
											\node at (0.8,0.5) {\color{red}$h_1$};
											\node at (1.1,-1.1) {\color{blue} $h_2$};
											\node at (5,2) {$h_3$};
											\end{tikzpicture}
										\end{center}
										\caption{The hypergraph in Example \ref{examplethm22}.}
									\end{figure}
									This  $\Gamma$ does not contain any closed system. Now, let $\gamma:H\rightarrow\mathbb{R}$ such that $\gamma(h_1):=\gamma(h_3):=\frac{1}{2}$ and $\gamma(h_2):=1$. Then $\gamma$ satisfies (\ref{mult0LH}), therefore $\mu_M^H=0$. 
								\end{ex}
				\begin{prop}\label{propindclosed}
					If $\Gamma$ has $l$ linearly independent closed systems, then
					\begin{equation*}
					\mu_M^H=\ldots=\mu_{M+1-l}^H=0,
					\end{equation*}i.e. the multiplicity of the eigenvalue $0$ for $L^H$ is at least $l$.
				\end{prop}
				\begin{proof}
					Let $h_1\ldots,h_M$ be the hyperedges of $\Gamma$. If $\mathcal{S}_1,\ldots,\mathcal{S}_l$ are linearly idependent closed systems, it means that the rows of the $l\times M$ matrix $A=(a_{ij})_{ij}$ such that
					\begin{equation*} 
					a_{ij}:=\begin{cases} 1 & \text{if }h_j\in\mathcal{S}_i\\ 0 & \text{otherwise} \end{cases} 
					\end{equation*}are linearly independent. Therefore, the functions $\gamma_i:H\rightarrow\mathbb{R}$ defined as $\gamma_i(h_j):=a_{ij}$ for each $i\in\{1,\ldots,l\}$ and each $j\in\{1,\ldots,M\}$ are linearly independent. Also, they all satisfy (\ref{mult0LH}). Therefore
					\begin{equation*}
					\mu_M^H=\ldots=\mu_{M+1-l}^H=0,
					\end{equation*}i.e. the multiplicity of the eigenvalue $0$ for $L^H$ is at least $l$.
				\end{proof}
				\begin{cor}
					If $\Gamma$ has $k$ pairwise disjoint closed systems, then
					\begin{equation*}
					\mu_M^H=\ldots=\mu_{M+1-k}^H=0,
					\end{equation*}i.e. the multiplicity of the eigenvalue $0$ for $L^H$ is at least $k$.
				\end{cor}
				\begin{proof}
					The claim follows from Prop. \ref{propindclosed} and from the fact that, if $\mathcal{S}_1,\ldots,\mathcal{S}_k$ are pairwise disjoint closed systems, then they are also linearly independent.
				\end{proof}
				\subsection{Independent hyperedges and independent vertices}
		We end the section about the eigenvalue $0$ by giving, with Prop. \ref{propker}, another characterization of $m_V$ and $m_H$. Before, we define the \emph{incidence matrix} of a hypergraph and we define \emph{linearly independence} for both hyperedges and vertices. 
\begin{definition}
	Let $\Gamma=(V,H)$ be a hypergraph with $N$ vertices $v_1,\ldots,v_N$ and $M$ hyperedges $h_1,\ldots,h_M$. We define the $N\times M$ \emph{incidence matrix} of $\Gamma$ as $\mathcal{I}:=(\mathcal{I}_{ij})_{ij}$, where
	\begin{equation*} 
	\mathcal{I}_{ij}:=\begin{cases} 1 & \text{if $v_i$ is an input and not an output of }h_j\\ -1 & \text{if $v_i$ is an output and not an input of }h_j\\ 0 & \text{otherwise.} \end{cases} 
	\end{equation*}Therefore each row $\mathcal{I}_i$ of $\mathcal{I}$ represents a vertex $v_i$ and each column $\mathcal{I}^j$ of $\mathcal{I}$ represents a hyperedge $h_j$.
\end{definition}
\begin{definition}
	Given $J\subseteq\{1,\ldots,M\}$, we say that the hyperedges $\{h_j\}_{j\in J}$ are \emph{linearly independent} if the corresponding columns in the incidence matrix are linearly independent, that is, if $\{\mathcal{I}^j\}_{j\in J}$ are linearly independent. Analogously, given $I\subseteq\{1,\ldots,N\}$, we say that the vertices $\{v_i\}_{i\in I}$ are \emph{linearly independent} if the corresponding rows in the incidence matrix are linearly independent, that is, if $\{\mathcal{I}_i\}_{i\in I}$ are linearly independent.
\end{definition}
\begin{rmk}
	Linear dependence of hyperedges means the following: we see each hyperedge as the sum of all its inputs minus the sum of all its outputs (and we can forget about the catalysts). If a hyperedge can be written as a linear combination of the other ones, with coefficients in $\mathbb{R}$, we talk about linear dependence. Analogously, in order to talk about linear dependence of vertices, we see each vertex as the sum of all the hyperedges in which it is an input minus the sum of all the hyperedges in which it is an output (and we can forget the hyperedges in which it is a catalyst).
\end{rmk}
\begin{prop}\label{propker}
	\begin{equation*}
	\dim(\ker \mathcal{I})=m_H \qquad \text{ and }\qquad\dim(\ker \mathcal{I}^\top)=m_V.
	\end{equation*}
\end{prop}
\begin{proof}
	Let's first observe that we can see a function $\gamma:H\rightarrow\mathbb{R}$ as a vector $(\gamma_1,\ldots,\gamma_M)\in\mathbb{R}^M$ such that $\gamma_j=\gamma(h_j)$. Also, two such functions are linearly independent if and only if the corresponding vectors are linearly independent. Now,
	\begin{align*}
	\mathcal{I}\cdot\begin{pmatrix} \gamma_1 \\ \vdots \\ \gamma_M \end{pmatrix}=\mathbf{0} &\iff \sum_{j=1}^M\mathcal{I}_{ij}\cdot\gamma_j=0\qquad \forall i\in\{1,\ldots,N\}\\
	&\iff \sum_{j_{\text{in}}: i\text{ input of }h_{j_{\text{in}}}}\gamma_{j_{\text{in}}}=\sum_{j_{\text{out}}: i\text{ output of }h_{j_{\text{out}}}}\gamma_{j_{\text{out}}} \qquad \forall i\in\{1,\ldots,N\}\\
	&\iff \gamma \text{ satisfies } (\ref{mult0LH})\\
	&\iff \gamma \text{ is an eigenfunction of $L^H$ with eigenvalue }0.
	\end{align*}Therefore
	\begin{equation*}
	\dim(\ker \mathcal{I})=m_H.
	\end{equation*}With an analogous proof, one can see that
	\begin{equation*}
	\dim(\ker \mathcal{I}^\top)=m_V.
	\end{equation*}
\end{proof}We shall now see four corollaries of Prop. \ref{propker}.
\begin{cor}
	$m_H$ and $m_V$ don't change if we replace a hyperedge $h$ containing a catalyst $v$ by the new hyperedge $h\setminus \{v\}$.
\end{cor}
\begin{proof}
	It follows from Prop. \ref{propker} and by the definition of $\mathcal{I}$.
\end{proof}
\begin{cor}\label{corlindep}
	If there are linearly dependent hyperedges, then $m_H>0$. If there are linearly dependent vertices, then $m_V>0$.
\end{cor}
\begin{cor}\label{cormH}
	\begin{equation*}
	m_H=M-\text{maximum number of linearly independent hyperedges}
	\end{equation*}and
	\begin{equation*}
	m_V=N-\text{maximum number of linearly independent vertices}.
	\end{equation*}
\end{cor}
\begin{proof}
	It follows by Prop. \ref{propker} and by the Rank-Nullity Theorem.
\end{proof}
\begin{cor}
	In the case of graphs, 
	\begin{enumerate}
		\item Edges are linearly dependent if and only if they form at least one cycle;
		\item Vertices are linearly dependent if and only if they cover at least one connected component of the graph.
	\end{enumerate}
\end{cor}
\begin{proof}
	\begin{enumerate}
		In order to prove (1), assume first that a set of edges forms a cycle given by \begin{equation*}
		e_1=[v_1,v_2], e_2=[v_2,v_3],\ldots, e_k=[v_k,v_1].
		\end{equation*}Then, if we consider the corresponding columns of the incidence matrix, it's clear that
		\begin{equation*}
		\mathcal{I}^1+\mathcal{I}^2+\ldots+\mathcal{I}^k=0.
		\end{equation*}Therefore any set of edges containing $e_1,\ldots,e_k$ is linearly dependent. Vice versa, assume that $e_1,\ldots,e_k$ are linearly dependent and let $\Gamma'$ be the graph given by these edges. Then, by Cor. \ref{corlindep}, $m'_H>0$. Since $m'_H$ is the number of cycles contained in $\Gamma'$, this implies that $e_1,\ldots,e_k$ form at least one cycle.\newline
		One can prove (2) in a similar way.
	\end{enumerate}
\end{proof}
\begin{rmk}
	Interestingly, the equation
	\begin{equation}\label{balancingequation}
	\mathcal{I}\cdot\begin{pmatrix} \gamma_1 \\ \vdots \\ \gamma_M \end{pmatrix}=\mathbf{0}
	\end{equation}for the eigenfunctions $\gamma$ of $L^H$ that have eigenvalue $0$, reminds of the \emph{metabolite balancing equation} in the \emph{metabolic pathway analysis} \cite{steady-state}. In this setting, the $v_i$'s are metabolites, the $h_j$'s are metabolic reactions and the incidence matrix $\mathcal{I}$ is replaced by the similar \emph{stoichiometric matrix}. With Equation (\ref{balancingequation}), in this case, one looks for a \emph{flux distribution} $(\gamma_1,\ldots,\gamma_M)$ such that each $\gamma_j$ describes the net rate of the corresponding reaction $h_j$ and such that, with this flux distribution, there is a balance between the metabolites which are consumed and the ones which are producted, in the overall stoichiometry. For this reason, Equation (\ref{balancingequation}) in this case is called \emph{metabolite balancing equation} and it describes the so-called \emph{pseudo steady-state}.
\end{rmk}

	\section{Applications of the Min-max Principle}\label{Section Applications of the Min-max Principle}
	In this section, we will apply the following theorem in order to get more insight about the spectra of $L^V$ and $L^H$.
	\begin{theorem}[Courant-Fischer-Weyl min-max principle]
		Let $V$ be an $N$-dimensional vector space with a positive definite scalar product $(.,.)$. Let $\mathcal{V}_k$ be the family of all $k$-dimensional subspaces of $V$. Let $A : V \rightarrow V$ be a self adjoint linear operator. Then the eigenvalues $\mu_1\geq\ldots\geq\mu_N$ of $A$ can be obtained by
		\begin{equation*}
		\mu_k=\min_{V_{N-k+1}\in\mathcal{V}_{N-k+1}}\max_{g(\neq 0)\in V_{N-k+1}}\frac{(Ag,g)}{(g,g)}=\max_{V_k\in\mathcal{V}_k}\min_{g(\neq 0)\in V_k}\frac{(Ag,g)}{(g,g)}.
		\end{equation*}
	The vectors $g_k$ realizing such a min-max or max-min then are corresponding eigenvectors, and the min-max spaces $\mathcal{V}_{N-k+1}$ are spanned by the eigenvectors for the eigenvalues $\mu_{N-k+1},\ldots,\mu_N$, and analogously, the max-min spaces $\mathcal{V}_k$ are spanned by the eigenvectors for the eigenvalues $\mu_1,\ldots,\mu_{N-k+1}$. Thus, we also have
	\begin{equation}\label{eqminmax}
	\mu_k=\min_{g\in V,(g,g_j)=0\text{ for }j=k+1,\ldots,N}\frac{(Ag,g)}{(g,g)}=\max_{g\in V,(g,g_l)=0\text{ for }l=1,\ldots,k-1}\frac{(Ag,g)}{(g,g)}.
	\end{equation}
In particular,
		\begin{equation*}
		\mu_1=\max_{g\in V}\frac{(Ag,g)}{(g,g)},\qquad \mu_N=\min_{g\in V}\frac{(Ag,g)}{(g,g)}.
		\end{equation*}
	\end{theorem}
	\begin{definition}
		 $\frac{(Ag,g)}{(g,g)}$ is called the \emph{Rayleigh quotient} of $g$.
	\end{definition}
	\begin{rmk}
		Without loss of generality, we may assume $(g,g)=1$ in (\ref{eqminmax}).
	\end{rmk}
	\subsection{Largest eigenvalue}\label{Largest eigenvalue}
	Since $L^V$ and $L^H$ are self-adjoint operators, we can apply the Courant-Fischer-Weyl min-max Principle and find, in particular, two alternative ways of computing $\mu_1$:
	\begin{enumerate}
		\item\label{lambda_N^V} \begin{align*}
		\mu_1&=\max_{f}\frac{(\delta f,\delta f)_H}{(f,f)_V}\\
		&=\max_{f}\frac{\sum_{h\in H}\delta f(h)^2}{\sum_{v\in V}\deg v\cdot f(v)^2}\\
		&=\max_{f:\sum_{v\in V}\deg v\cdot f(v)^2=1}\sum_{h\in H}\delta f(h)^2\\
		&=\max_{f:\sum_{v\in V}\deg v\cdot f(v)^2=1}\sum_{h\in H}\biggl(\sum_{v_i \text{ input of }h}f(v_i)-\sum_{v^j \text{ output of }h}f(v^j)\biggr)^2
		\end{align*}and
		\item\label{lambda_N^H} \begin{align*}
		\mu_1&=\max_{\gamma}\frac{(\delta^* \gamma,\delta^* \gamma)_V}{(\gamma,\gamma)_H}\\
		&=\max_{\gamma}\frac{\sum_{v\in V}\deg v\cdot \delta^* \gamma(v)^2}{\sum_{h\in H}\gamma(h)^2}\\
		&=\max_{\gamma:\sum_{h\in H}\gamma(h)^2=1}\sum_{v\in V}\deg v\cdot \delta^* \gamma(v)^2\\
		&=\max_{\gamma:\sum_{h\in H}\gamma(h)^2=1}\sum_{v\in V}\frac{1}{\deg v}\cdot \biggl(\sum_{h_{\text{in}}: v\text{ input}}\gamma(h_{\text{in}})-\sum_{h_{\text{out}}: v\text{ output}}\gamma(h_{\text{out}})\biggr)^2.
		\end{align*}
		\end{enumerate}
		\begin{ex}\label{ex1}
			Consider a hypergraph with only one hyperedge $h$ that involves $N$ vertices: $k$ inputs and $m$ outputs, with $N\leq k+m\leq 2N$, so that there are $k+m-N$ catalysts. Then
			\begin{align*}
			\mu_1&=\max_{\gamma:\gamma(h)^2=1}\sum_{v\in V} \biggl(\sum_{h: v\text{ input}}\gamma(h)-\sum_{h: v\text{ output}}\gamma(h)\biggr)^2\\
			&=\big|\text{inputs that are not outputs}\big|+\big|\text{outputs that are not inputs}\big|\\
			&=\big|\text{inputs}\big|+\big|\text{outputs}\big|-2\cdot\big|\text{catalysts}\big|\\
			&=k+m-2k-2m+2N\\
			&=2N-k-m.
			\end{align*}In particular, this is the only eigenvalue of $L^H$. Observe that $\mu_1$ is equal to $0$ if and only if $2N=k+m$, i.e. if and only if all vertices are catalysts, while $\mu_1$ achieves the largest value $N$ exactly when $k+m=N$, i.e. when there are no catalysts.
		\end{ex}
			\begin{rmk}
				The previous example implies that $\mu_1$ can not be bounded from above by a quantity that does not depend on the number of vertices $N$ (while, for graphs, we always have $\mu_1\leq 2$). One should also compare this with Prop. \ref{charlambdaN}. 
			\end{rmk}

				\subsubsection{Bipartite hypergraphs}
				We know that the following theorem holds for graphs:
				\begin{theorem}\label{bipartitegraph}
				Let $\Gamma$ be a graph. Then $\mu_1\leq 2$ and the equality holds if and only if $\Gamma$ is bipartite.
				\end{theorem}
				\begin{recall}
				Recall that a graph is \emph{bipartite} if one can decompose the vertex set as a disjoint union $V=V_1\sqcup V_2$ such that every edge has one of its endpoints in $V_1$ and the other in $V_2$.
				\end{recall}We will now generalize the notion of bipartite graph and extend it to hypergraphs, then we will generalize Theorem \ref{bipartitegraph}.
				\begin{definition}
				We say that a hypergraph $\Gamma$ is \emph{bipartite} if one can decompose the vertex set as a disjoint union $V=V_1\sqcup V_2$ such that, for every hyperedge $h$ of $\Gamma$, either $h$ has all its inputs in $V_1$ and all its outputs in $V_2$, or vice versa.
				\end{definition}
				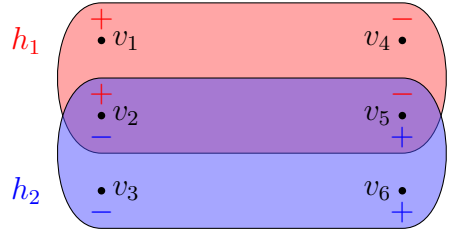
\begin{figure}[h]
					\begin{center}
\begin{tikzpicture}
\node (v3) at (1,0) {};
\node (v2) at (1,1) {};
\node (v1) at (1,2) {};
\node (v6) at (5,0) {};
\node (v5) at (5,1) {};
\node (v4) at (5,2) {};

\begin{scope}[fill opacity=0.5]
\filldraw[fill=red!70] ($(v1)+(0,0.5)$) 
to[out=180,in=180] ($(v2) + (0,-0.5)$) 
to[out=0,in=180] ($(v5) + (0,-0.5)$)
to[out=0,in=0] ($(v4) + (0,0.5)$)
to[out=180,in=0] ($(v1)+(0,0.5)$);
\filldraw[fill=blue!70] ($(v2)+(0,0.5)$) 
to[out=180,in=180] ($(v3) + (0,-0.5)$) 
to[out=0,in=180] ($(v6) + (0,-0.5)$)
to[out=0,in=0] ($(v5) + (0,0.5)$)
to[out=180,in=0] ($(v2)+(0,0.5)$);
\end{scope}

\fill (v1) circle (0.05) node [right] {$v_1$} node [above] {\color{red}$+$};
\fill (v2) circle (0.05) node [right] {$v_2$} node [above] {\color{red}$+$} node [below] {\color{blue}$-$};
\fill (v3) circle (0.05) node [right] {$v_3$} node [below] {\color{blue}$-$};
\fill (v4) circle (0.05) node [left] {$v_4$} node [above] {\color{red}$-$};
\fill (v5) circle (0.05) node [left] {$v_5$} node [above] {\color{red}$-$}node [below] {\color{blue}$+$};
\fill (v6) circle (0.05) node [left] {$v_6$} node [below] {\color{blue}$+$};

\node at (0,2) {\color{red}$h_1$};
\node at (0,0) {\color{blue}$h_2$};
\end{tikzpicture}
					\end{center}
					\caption{A bipartite hypergraph with $V_1=\{v_1,v_2,v_3\}$ and $V_2=\{v_4,v_5,v_6\}$.}\label{bipartiteh}
				\end{figure}
				\begin{rmk}
					It is clear from the definition that:\begin{itemize}
						\item If a hypergraph is bipartite it does not contain catalysts;
						\item The definition of bipartite hypergraph applied to graphs gives exactly the definition of bipartite graph that we already know.
					\end{itemize}
				\end{rmk}
				\begin{lem}\label{|h|=N}
					Let $\Gamma$ be a bipartite hypergraph. Then
					\begin{equation*}
					\mu_1\geq \frac{\sum_{h\in H}\big|h\big|^2}{\sum_{h\in H}\big|h\big|}.
					\end{equation*}
				\end{lem}
				\begin{proof}
					Since $\Gamma$ is bipartite, we can write
					\begin{align*}
					\mu_1&=\max_{f:\sum_{v\in V}\deg v\cdot f(v)^2=1}\sum_{h\in H}\biggl(\sum_{v_i \text{ input of }h}f(v_i)-\sum_{v^j \text{ output of }h}f(v^j)\biggr)^2\\
					&=\max_{f:\sum_{v\in V}\deg v\cdot f(v)^2=1}\sum_{h\in H}\biggl(\sum_{v_i \in h:f(v_i)>0}f(v_i)-\sum_{v^j \in h:f(v^j)<0}f(v^j)\biggr)^2.
					\end{align*} 
					Now, let \begin{equation*}
					f(v):=\frac{1}{\sqrt{\sum_v\deg v}}
					\end{equation*}for every $v\in V_1$ and
					\begin{equation*}
					f(w):=-\frac{1}{\sqrt{\sum_w\deg w}}
					\end{equation*}for every $w \in V_2$.
					Then
				    \begin{equation*}
					\sum_{v\in V}\deg v\cdot f(v)^2=1
					\end{equation*}and
					\begin{align*}
					&\sum_{h\in H}\biggl(\sum_{v_i \in h:f(v_i)>0}f(v_i)-\sum_{v^j \in h:f(v^j)<0}f(v^j)\biggr)^2\\&=\sum_{h\in H}\biggl(\frac{1}{\sqrt{\sum_v\deg v}}\cdot \big|h\big|\biggr)^2\\
					&=\frac{\sum_{h\in H}\big|h\big|^2}{\sum_{h\in H}\big|h\big|},
					\end{align*}where the last equality is due to the fact that $\sum_v\deg v=\sum_{h\in H}\big|h\big|$.\newline
					Therefore
						\begin{equation*}
						\mu_1\geq \frac{\sum_{h\in H}\big|h\big|^2}{\sum_{h\in H}\big|h\big|}.
						\end{equation*}
				\end{proof}
				\begin{rmk}
					The quantity 	\begin{equation*}
					h':=\frac{\sum_{h\in H}\big|h\big|^2}{\sum_{h\in H}\big|h\big|}
					\end{equation*}appearing in Lemma \ref{|h|=N} has the biggest value $N$ exactly when every $h\in H$ has the biggest possible cardinality, which is $N$.
				\end{rmk}
				\begin{rmk}
				Recall from Example \ref{ex1} that, for bipartite hypergraphs with only one hyperedge, $\mu_1=N$, therefore in this case $\mu_1=h'$.
				\end{rmk}
				\begin{rmk}
					\end{rmk}Let's apply Lemma \ref{|h|=N} to a bipartite graph $\Gamma$. Since $\big|e\big|=2$ for every edge, the lemma tells us that 
					\begin{equation*}
					\mu_1\geq\frac{\sum_{e\in E}4}{\sum_{e\in E}2}=\frac{\big|E\big|\cdot 4}{\big|E\big|\cdot 2}=2
					\end{equation*}and, as we know, this is actually an equality.
				
				\begin{prop}\label{bipartitehyp}
					Let $\Gamma$ be a hypergraph with largest eigenvalue $\mu_1$. Then
					\begin{equation*}
					\mu_1\leq \mu'_1
					\end{equation*}
                where $\mu_1'$ is the largest eigenvalue of a bipartite hypergraph that has the same number of hyperedges as $\Gamma$ and also the same number of inputs and the same number of outputs in each hyperedge (catalysts are not included).\newline
                The equality holds if and only if $\Gamma$ is bipartite.
				\end{prop}
				\begin{proof}
Let $\Gamma$ be a hypergraph with largest eigenvalue $\mu_1$. Then
				\begin{align*}
				\mu_1&=\max_{f:\sum_{v\in V}\deg v\cdot f(v)^2=1}\sum_{h\in H}\biggl(\sum_{v_i \text{ input of }h}f(v_i)-\sum_{v^j \text{ output of }h}f(v^j)\biggr)^2\\
				&\leq\max_{f:\sum_{v\in V}\deg v\cdot f(v)^2=1}\sum_{h\in H}\biggl(\sum_{v_i \in h:f(v_i)>0}f(v_i)-\sum_{v^j \in h:f(v^j)<0}f(v^j)\biggr)^2,
				\end{align*}where the last inequality is due to the fact that, for every $f$,
				\begin{align*}
				&\biggl|\sum_{v_i \text{ input of }h}f(v_i)-\sum_{v^j \text{ output of }h}f(v^j)\biggr|\\
				&\leq\biggl|\sum_{v_i \in h:f(v_i)>0}f(v_i)-\sum_{v^j \in h:f(v^j)<0}f(v^j)\biggr|.
				\end{align*}
				It is clear that the inequality for $\mu_1$ becomes a inequality if and only if, for every $h\in H$, we can let such $f$ be positive in the inputs and negative in the outputs, or vice versa. And this is possible if and only if the hypergraph is bipartite. Therefore \begin{equation*}
				\mu_1\leq \mu_1'
				\end{equation*}and the equality holds if and only if $\Gamma$ is bipartite.
					\end{proof}
					\begin{rmk}
						We can put together Lemma \ref{|h|=N} and Prop. \ref{bipartitehyp} and say that the largest value of $\mu_1$ is achieved by bipartite hypergraphs and that, in this case, $\mu_1\geq h'$. In particular, $\mu_1\geq h'$ becomes an equality for both bipartite graphs and bipartite hypergraphs with only one hyperedge. But it is in general not an equality, as proved by the next example.
					\end{rmk}
\begin{ex}\label{ex2,7}
	Let $\Gamma=(\{v_1,v_2,v_3,v_4\},\{h_1,h_2\})$ be the bipartite hypergraph such that:
	\begin{enumerate}
		\item $h_1$ has $v_1$ and $v_2$ as inputs and $v_3$ as output;
		\item $h_2$ has $v_1$ as input and $v_4$ as output.
	\end{enumerate}
			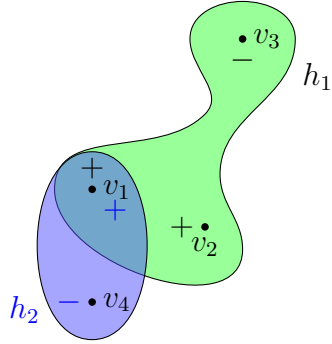
\begin{figure}[h]
				\begin{center}
					\begin{tikzpicture}
					\node (v4) at (4,2.5) {};
					\node (v1) at (2,0.5) {};
					\node (v2) at (3.5,0) {};
					\node (v3) at (2.0,-1) {};
					
					\begin{scope}[fill opacity=0.5]
					\filldraw[fill=green!80] ($(v1)+(-0.5,0)$)
					to[out=90,in=225] ($(v4)+(-0.5,-1)$)
					to[out=45,in=270] ($(v4)+(-0.7,0)$)
					to[out=90,in=180] ($(v4)+(0,0.5)$)
					to[out=0,in=90] ($(v4)+(0.7,0)$)
					to[out=270,in=90] ($(v4)+(-0.3,-1.8)$)
					to[out=270,in=90] ($(v2)+(0.5,-0.3)$)
					to[out=270,in=270] ($(v1)+(-0.5,0)$);
					\filldraw[fill=blue!70] ($(v1)+(0,0.5)$) 
					to[out=0,in=0] ($(v3) + (0,-0.5)$)
					to[out=180,in=180] ($(v1)+(0,0.5)$);
					\end{scope}
					
					\fill (v1) circle (0.05) node [right] {$v_1$} node [below right] {\color{blue}$+$} node [above] {$+$};
					\fill (v2) circle (0.05) node [below] {$v_2$} node [left] {$+$};
					\fill (v3) circle (0.05) node [right] {$v_4$} node [left] {\color{blue}$-$};
					\fill (v4) circle (0.05) node [right] {$v_3$} node [below] {$-$};
					
					\node at (1.1,-1.1) {\color{blue} $h_2$};
					\node at (5,2) {$h_1$};
					\end{tikzpicture}
				\end{center}
				\caption{The hypergraph in Example \ref{ex2,7}.}
			\end{figure}
			In this case,
			\begin{equation*}
			h'=\frac{\sum_{h\in H}\big|h\big|^2}{\sum_{h\in H}\big|h\big|}=\frac{13}{5}=2,6.
			\end{equation*}
			Now, let's compute $\mu_1$ using the Min-max Principle applied to $L^H$. For simplicity, let $\gamma(h_1):=x$ and let $\gamma(h_2):=y$. Then
			\begin{align*}
			\mu_1&=\max_{\gamma:\sum_{h\in H}\gamma(h)^2=1}\sum_{v\in V}\frac{1}{\deg v}\cdot \biggl(\sum_{h_{\text{in}}: v\text{ input}}\gamma(h_{\text{in}})-\sum_{h_{\text{out}}: v\text{ output}}\gamma(h_{\text{out}})\biggr)^2\\
			&=\max_{x,y\in\mathbb{R}:x^2+y^2=1}\biggl(x^2+x^2+\frac{(x+y)^2}{2}+y^2\biggr)\\
			&=\max_{x,y\in\mathbb{R}:x^2+y^2=1}\biggl(\frac{3}{2}+x^2+xy\biggr),
			\end{align*}where in the last equality we have used the fact that $x^2+y^2=1$. Now, let $x:=\cos(t)$ and let $y:=\sin (t)$. Then
			\begin{equation*}
			\mu_1=\max_{0\leq t\leq 2\pi}\biggl(\frac{3}{2}+\cos^2(t)+\cos(t)\cdot\sin(t)\biggr).
			\end{equation*}Now,
			\begin{equation*}
			\frac{d}{dt}\biggl(\frac{3}{2}+\cos^2(t)+\cos(t)\cdot\sin(t)\biggr)=\cos(2t)-\sin(2t),
			\end{equation*}which has value $0$ for $t=\frac{\pi}{8}$ and $t=\frac{5\pi}{8}$. In particular, for $t=\frac{5\pi}{8}$ we get that
			\begin{align*}
			\mu_M^H&=\frac{3}{2}+\cos^2\biggl(\frac{5\pi}{8}\biggr)+\cos\biggl(\frac{5\pi}{8}\biggr)\cdot\sin\biggl(\frac{5\pi}{8}\biggr)\\
			&=2-\frac{1}{\sqrt{2}}\\
			&\cong 1,29.
			\end{align*}For $t=\frac{\pi}{8}$ we get that 
			\begin{align*}
			\mu_1&=\frac{3}{2}+\cos^2\biggl(\frac{\pi}{8}\biggr)+\cos\biggl(\frac{\pi}{8}\biggr)\cdot\sin\biggl(\frac{\pi}{8}\biggr)\\
			&=2+\frac{1}{\sqrt{2}}\\
			&\cong 2,71.
			\end{align*}
			In particular, $\mu_1>h'$. This proves that the $\geq$ of Lemma \ref{|h|=N} is, in general, not an equality.
			\end{ex}
			Let's end this section by proving that there is another family of bipartite hypergraphs with $\mu_1=h'$.
			\begin{lem}
			Let $\Gamma$ be a bipartite graph on $N$ nodes such that $\big|h\big|=N$ for every $h\in H$. Then
			\begin{equation*}
			\mu_1=h'=N.
			\end{equation*}
			\end{lem}
			\begin{proof}Let's first observe that, in this case,
				\begin{equation*}
				h'=\frac{\sum_{h}\big|h\big|^2}{\sum_{h}\big|h\big|}=\frac{\big|H\big|\cdot N^2}{\big|H\big|\cdot N}=N.
				\end{equation*}
			Now observe that, for any bipartite hypergraph,
			\begin{align*}
			\mu_1&=\max_{f}\frac{\sum_{h\in H}\biggl(\sum_{v_i \text{ input of }h}f(v_i)-\sum_{v^j \text{ output of }h}f(v^j)\biggr)^2}{\sum_{v\in V}\deg v\cdot f(v)^2}\\
			&=\max_{f}\frac{\sum_{h\in H}\biggl(\sum_{v_i \in h:f(v_i)>0}f(v_i)-\sum_{v^j \in h:f(v^j)<0}f(v^j)\biggr)^2}{\sum_{v\in V}\deg v\cdot f(v)^2}\\
			&=\max_{f>0}\frac{\sum_{h\in H}\biggl(\sum_{v\in h}f(v)\biggr)^2}{\sum_{v\in V}\deg v\cdot f(v)^2}.
			\end{align*}In our particular case, since $\{v\in h\}=\{v\in V\}$ for every $h$ and since $\deg_v=\big|H\big|$ for every $v$, we have that
			\begin{align*}
			\mu_1&=\max_{f>0}\frac{\big|H\big|\cdot\biggl(\sum_{v\in V}f(v)\biggr)^2}{\big|H\big|\cdot\sum_{v\in V}\cdot f(v)^2}\\
			&=\max_{f>0}\frac{\biggl(\sum_{v\in V}f(v)\biggr)^2}{\sum_{v\in V} f(v)^2}\\
			&=\mu_1',
			\end{align*}where $\mu_1'$ is the largest eigenvalue of a bipartite hypergraph on $N$ nodes with only one hyperedge. As we have seen in Example \ref{ex1}, $\mu_1'=N$, therefore $\mu_1=h'=N$.
			\end{proof}

				\section{Isospectral hypergraphs}\label{Isospectral hypergraphs}
				We already know that two graphs cannot always be distinguished by their spectra, but the spectrum reveals  some important properties. \emph{Is the graph bipartite? How many connected components does it have? How many cycles? Is it complete?} -- these are all questions that can be answered using the spectrum of the Laplace operator for graphs, so even if it does not distinguish the details of graphs, it does partition them into important families. We expect something similar to happen for hypergraphs.

For instance, the spectrum of $L^V$ of all complete bipartite graphs with the same number of vertices is the same. (The multiplicity of the eigenvalue $0$ of $L^H$, however, distinguishes between them.) For hypergraphs, a new phenomenon arises. 
					\begin{lem}\label{lemisoL^H}
						The spectrum of $L^V$ and $L^H$ doesn't change if we reverse the role of a vertex in all the hyperedges in which it is contained, i.e. if we let it become an input where it is an output and we let it become an output where it is an input.
					\end{lem}
					\begin{proof}By the Min-max Principle, the spectrum of $L^H$ is given by the \emph{min-max} of the Rayleigh quotient, which is now
						\begin{equation*}
									\frac{\sum_{v\in V}\frac{1}{\deg v}\cdot \biggl(\sum_{h_{\text{in}}: v\text{ input}}\gamma(h_{\text{in}})-\sum_{h_{\text{out}}: v\text{ output}}\gamma(h_{\text{out}})\biggr)^2}{\sum_{h\in H}\gamma(h)^2}.
						\end{equation*}Now, since for each $v\in V$ we have
						\begin{equation*}
						\biggl(\sum_{h_{\text{in}}: v\text{ input}}\gamma(h_{\text{in}})-\sum_{h_{\text{out}}: v\text{ output}}\gamma(h_{\text{out}})\biggr)^2=\biggl(\sum_{h_{\text{out}}: v\text{ output}}\gamma(h_{\text{out}})-\sum_{h_{\text{in}}: v\text{ input}}\gamma(h_{\text{in}})\biggr)^2,
						\end{equation*}the Rayleigh quotient and therefore the spectrum of $L^H$ (and $L^V$) doesn't change if we reverse the role of a vertex in all the hyperedges in which it is contained.
					\end{proof}
\begin{ex}
Let $\Gamma=(V,E)$ be a connected graph. Lemma \ref{lemisoL^H} tells us that, if we reverse the role of a vertex $v\in V$ in all the edges in which it is contained, the spectrum of $\Gamma$ doesn't change. This transformation actually creates an oriented graph where all edges that have $v$ as an endpoint have either two inputs or two outputs. But this situation is not interesting from both the chemical point of view (where we assume that there are always both inputs and outputs) and the mathematical point of view, because in graph theory one always assigns an orientation to an edge by choosing exactly one input and exactly one output. Therefore, in order to have consistency with our theory, we should assume that every time we apply the operation described in Lemma \ref{lemisoL^H} to a vertex $v$, we also apply it to all its neighbors. For the same reason, we should also apply it to the neighbors of its neighbors and therefore, by induction, since we are assuming that $\Gamma$ is connected, we should apply this operation to all vertices of $\Gamma$. In conclusion, Lemma \ref{lemisoL^H} in the case of graphs tells us that the spectrum doesn't change if we reverse the orientation of every edge in a given connected component.
\end{ex}

				\end{document}